%% file: main.tex
\documentclass{svproc}
\usepackage{url}

\usepackage{hyperref}
\usepackage{type1cm}        
\usepackage{makeidx}         
\usepackage{graphicx}        

\usepackage{multicol}        
\usepackage[bottom]{footmisc}

\usepackage{newtxtext}       %
\usepackage[varvw]{newtxmath}       

\usepackage{bm}
\usepackage{siunitx}

\newcounter{algorithm}

\input{macros}

\usepackage{scrextend} 

\begin{document}

\mainmatter              
\title{Quasi-Monte Carlo for Bayesian Shape Inversion Governed by the Poisson Problem Subject to Gevrey Regular Domain Deformations}
\titlerunning{Quasi-Monte Carlo for Bayesian Shape Inversion}  
%
\author{Ana Djurdjevac\inst{1,2} \and Vesa Kaarnioja\inst{2,3}\and
Max Orteu\inst{2}  \and Claudia Schillings\inst{2} }
\authorrunning{Ana Djurdjevac et al.} 
%
\tocauthor{Ana Djurdjevac, Vesa Kaarnioja, Max Orteu, and Claudia Schillings}
\institute{Mathematical Institute, University of Oxford, Woodstock Road, Oxford, OX2 6GG, UK\\
{\tt ana.djurdjevac@maths.ox.ac.uk}
\and
Department of Mathematics and Computer Science, Free University of Berlin, Arnimallee 6, 14195 Berlin, Germany,\\
{\tt m.orteu.capdevila@fu-berlin.de, c.schillings@fu-berlin.de}
\and
School of Engineering Sciences, LUT University, P.O.~Box 20, 53851 Lappeenranta, Finland,\\
{\tt vesa.kaarnioja@iki.fi}}

\maketitle              

\begin{abstract}
We consider the application of a quasi-Monte Carlo cubature rule to Bayesian shape inversion subject to the Poisson equation under Gevrey regular parameterizations of domain uncertainty. We analyze the parametric regularity of the associated posterior distribution and design randomly shifted rank-1 lattice rules which can be shown to achieve dimension-independent, faster-than-Monte Carlo cubature convergence rates for high-dimensional integrals over the posterior distribution. In addition, we consider the effect of dimension truncation and finite element discretization errors for this model. Finally, a series of numerical experiments are presented to validate the theoretical results.
\end{abstract}

\section{Introduction}

Shape recovery problems involve reconstructing the geometry or boundary of an object from indirect or incomplete measurements. They are a significant subset of inverse problems, in particular of those governed by partial differential equations (PDEs), e.g., in groundwater flow~\cite{DjKaOrSc:darcy}, electrical impedance tomography~\cite{DjKaOrSc:somersalo}, or the classical problem of reconstructing the shape of a domain based on the spectrum of a partial differential operator~\cite{DjKaOrSc:kac}. Indeed, in these instances, the geometry required for computational inversion is often not perfectly known. In the Bayesian statistical inversion paradigm~\cite{DjKaOrSc:kaipiosomersalo,DjKaOrSc:stuart}, it is possible to model the uncertain geometry as a random field in addition to other unknown model parameters of interest. Based on partial, indirect, and possibly noisy measurements of the state of a PDE system, 
the Bayesian approach enables making inferences about the uncertain geometry.

We shall investigate a setting where the measurement model can be described by the solution $u(\cdot,\omega)\!:D(\omega)\to \mathbb R$ to the Poisson problem
\begin{align}\label{DjKaOrSc:eq:poissonprob}
\begin{cases}
-\Delta u(\bsx,\omega)=f(\bsx)&\text{for}~\bsx\in D(\omega),\\
u(\bsx,\omega)=0&\text{for}~\bsx\in\partial D(\omega)
\end{cases}\quad\text{for $\mathbb P$-a.e.}~\omega\in \Omega,
\end{align}
where $(\Omega,\mathcal A,\mathbb P)$ is a probability space, $f$ is a fixed source term, and the domain $D(\omega)\subset\mathbb R^d$, $d\in\{1,2,3\}$, is assumed to be {\textit{uncertain}}. Analyses on effective approximation of the response statistics of this model problem have been carried out previously in~\cite{DjKaOrSc:canote16,DjKaOrSc:ChDjEl20,DjKaOrSc:djurdjevac24}. 

In this paper, we shall focus on the Bayesian inverse problem of inferring the domain shape based on measurements of certain observable quantities of interest (QoIs) of the Poisson problem~\eqref{DjKaOrSc:eq:poissonprob}. We shall model the domain uncertainty using the \emph{domain mapping method}~\cite{DjKaOrSc:harbrecht16,DjKaOrSc:xiu06}, where the uncertain domains $D(\omega)$ are modeled as images of a fixed reference domain $D_{\rm ref}\subset\mathbb R^d$ under a mapping $\boldsymbol V(\cdot,\bsy(\omega))\!:\mathbb R^d\to \mathbb R^d$, parameterized by a countable sequence of i.i.d.~random variables $\bsy=\bsy(\omega)$. It is standard practice to treat the random variables $\bsy$ as parameters supported over a sequence space $\varnothing\neq U\subset\mathbb R^{\mathbb N}$ equipped with an appropriate probability measure, a framework that we will also adopt.

The domain mapping $\boldsymbol V$ can be regarded in the Bayesian framework as a pushforward probability measure for the uncertain domain, and in this work, we shall develop a quasi-Monte Carlo (QMC) convergence analysis for a general class of Gevrey regular parameterizations for the domain mapping $\boldsymbol V$ in the inverse setting. The Gevrey class contains smooth, but not necessarily holomorphic, functions with a growth condition on the higher-order partial derivatives, and recently there has been a surge of interest in forward uncertainty quantification for PDEs with parametric inputs belonging to this class~\cite{DjKaOrSc:chernov1,DjKaOrSc:chernov2,DjKaOrSc:gk24gevrey,DjKaOrSc:harbrecht24}. Meanwhile, QMC analysis for Bayesian inference of a Gevrey regular parameterization of an unknown conductivity field has been considered by~\cite{DjKaOrSc:BaKaLa24} within the context of electrical impedance tomography.

A number of shape recovery problems subject to PDEs have been considered in the literature: inverse random acoustic scattering has been studied within the context of multilevel Halton and sparse grid cubatures~\cite{DjKaOrSc:dolz}, the well-posedness of Bayesian shape inversion for time-harmonic Helmholtz transmission and exterior Dirichlet problems has been studied by~\cite{DjKaOrSc:scarabosio24}, and the parametric regularity of the Bayesian posterior for PDEs subject to holomorphic random perturbation fields has been studied by~\cite{DjKaOrSc:gantnerpeters18} within the context of higher-order QMC. Hyv\"onen et al.~\cite{DjKaOrSc:hyvonen17} studied shape recovery for electrical impedance tomography numerically using sparse grids. 

The aim of the present work lies in developing cubature rules with fast rates of convergence for the computation of the posterior mean for Bayesian shape recovery problems in the setting described above. Our main contributions are the following: 
\begin{itemize} 
    \item We consider Bayesian shape inversion subject to Gevrey regular parameterizations of the input random field. Gevrey regular random fields cover a wider range of potential parameterizations for uncertain domains than those previously covered by  holomorphic models of domain uncertainty.
    \item We design dimension-independent randomly shifted rank-1 lattice QMC rules for this model and prove that they exhibit essentially linear cubature convergence rates with respect to the number of cubature points. 
    \item We present a series of high-dimensional numerical experiments which showcase our theoretical convergence results and the quality of the reconstructed features. In particular, we look at the convergence of the root mean square (rms) error and the reconstruction of the expected domain. 
\end{itemize}

This paper is organized as follows. The  notation and function spaces used throughout this work are introduced in Section~\ref{DjKaOrSc:sec:notations}. In Section~\ref{DjKaOrSc:sec: QMC for Bayesian inversion} we present the model problem and associated modeling assumptions, as well as the Bayesian approach to inverse problems and the basic properties of randomly shifted rank-1 lattice rules. Section~\ref{DjKaOrSc:sec:parametricanalysis} contains the main parametric regularity analysis developed for the Bayesian shape inversion problem, followed by our main convergence results. Numerical experiments are presented in Section~\ref{DjKaOrSc:sec:numex} to assess the sharpness of our theoretical results. The paper concludes with a summary of our findings in Section~\ref{DjKaOrSc:sec:conclusions}.

\subsection{Notation}\label{DjKaOrSc:sec:notations}

We will use boldfaced symbols to denote vectors and multi-indices, while the subscript notation $\nu_j$ is used to refer to the $j^{\rm th}$ component of $\boldsymbol \nu$. The set of finitely supported multi-indices is denoted by
$\mathscr F:=\{\boldsymbol\nu\in\mathbb N_0^{\mathbb N}:|\bsnu|<\infty\}$, where the {\em modulus} is defined by $|\bsnu|:=\sum_{j\geq 1}\nu_j$.

Let $\bsnu,\boldsymbol m\in\mathscr F$ and let $\bsx:=(x_j)_{j\geq 1}$ be a sequence of real numbers. We define the shorthand notations (with the convention $0^0:=1$):
\begin{align*}
\partial_{\bsy}^{\bsnu}:=\prod_{j\geq 1}\frac{\partial^{\nu_j}}{\partial y_j^{\nu_j}},\quad \binom{\bsnu}{\boldsymbol m}:=\prod_{j\geq 1}\binom{\nu_j}{m_j},\quad \text{and}\quad \bsx^{\boldsymbol \nu}:=\prod_{j\geq 1}x_j^{\nu_j}.
\end{align*}

For   a nonempty Lipschitz domain $D\subset\mathbb R^d$, $d\in\{1,2,3\}$, we denote by $H_0^1(D)$ the subspace of $H^1(D)$ with zero trace on $\partial D$, equipped with the norm $\|v\|_{H_0^1(D)}:=\|\nabla v\|_{L^2(D)}$. Moreover, we define 
$$
\|v\|_{L^\infty(D)}:={\rm ess\,sup}_{\bsx\in D}\|v(\bsx)\|,
$$
where $\|v\|$ is the absolute value if $v: D \to \mathbb{R}$, the Euclidean norm for vectors if $v: D \to \mathbb{R}^d$, and the spectral norm for matrices if $v: D \to \mathbb{R}^{d\times d}$, and
$$
\|v\|_{W^{1,\infty}(D)}:= \max\big\{{\rm ess\,sup}_{\bsx\in D}\|v(\bsx)\|,{\rm ess\,sup}_{\bsx\in D}\|\nabla v(\bsx)\|\big\},
$$
where $\nabla v$ is the gradient if $v$ is scalar-valued and the Jacobian matrix if $v$ is vector-valued. Finally, we also define the norm
$$
\|v\|_{\mathcal C^k(\overline{D})}:=\max_{|\bsnu|\leq k}\sup_{\bsx\in \overline{D}}|\partial_{\bsx}^{\bsnu}v(\bsx)|\quad\text{for}~v\in \mathcal C^k(\overline{D}).
$$

When working with spaces of parameters and their truncated counterparts, we in general write, e.g., $\bsy \in [-1/2, 1/2]^{\mathbb{N}}$ or $\bsy \in [-1/2, 1/2]^{s}$, for some $s \in \mathbb{N}$. Wherever this might lead to confusion, we use the subindex $\bsy_s$ for the truncated case. 

\section{Quasi-Monte Carlo Methods for Bayesian Shape Inversion Problems}
\label{DjKaOrSc:sec: QMC for Bayesian inversion}

In this section, we present the model problem and its variational formulation, and introduce the associated Bayesian inverse problem. Finally, we outline rank-1 QMC methods for the problem.

\subsection{Model Problem}%

Let $D_{\rm ref}\subset\mathbb R^d$, $d\in\{1,2,3\}$, denote a fixed, nonempty, and bounded Lipschitz domain and let $U:=[-1/2,1/2]^{\mathbb N}$ be a set of parameters, which will later be truncated to some stochastic dimension $s$. We denote by $\boldsymbol V(\cdot,\bsy)\!:\overline{D_{\rm ref}}\to\mathbb R^d$ a domain mapping for $\bsy\in U$. Furthermore, we denote the Jacobian matrix of $\boldsymbol V(\bsx,\bsy)$ with respect to $\bsx\in D_{\rm ref}$ by $J(\cdot,\bsy)\!:D_{\rm ref}\to\mathbb R^{d\times d}$ for $\bsy\in U$. The family of {\em admissible domains} $\{D(\bsy)\}_{\bsy\in U}$ is defined by setting
$$
D(\bsy):=\boldsymbol V(D_{\rm ref},\bsy),\quad \bsy\in U,
$$
and the {\em hold-all domain} $\mathscr D$ is defined as
$$
\mathscr D:=\bigcup_{\bsy\in U}D(\bsy).
$$

\smallskip

\noindent \textbf{Assumptions on transformation $\bsV$ and source term $f$}
\begin{addmargin}[1.3em]{0em}
\begin{enumerate}
\item[(A1)] For each $\bsy\in U$, $\boldsymbol V(\cdot,\bsy)\!:\overline{D_{\rm ref}}\to\mathbb R^d$ is a $\mathcal C^1$-diffeomorphism and there exists a constant $C\geq 1$ independent of $\bsy$ such that%
$$
\|\boldsymbol V(\cdot,\bsy)\|_{\mathcal C^1(\overline{D_{\rm ref}})}\leq C\quad \text{and}\quad \|\boldsymbol V^{-1}(\cdot,\bsy)\|_{\mathcal C^1(\overline{D(\bsy)})}\leq C\quad\text{for all}~\bsy\in U.
$$
\item[(A2)] Gevrey regularity of $\boldsymbol V$: there exist constants $C,\beta\geq 1$ and a sequence of nonnegative numbers $\boldsymbol b=(b_j)_{j\geq 1}$ such that%
$$
\|\partial_\bsy^{\bsnu}\boldsymbol V(\cdot,\bsy)\|_{W^{1,\infty}(D_{\rm ref})}\leq C(|\bsnu|!)^{\beta}\boldsymbol b^{\bsnu}\quad\text{for all}~\bsnu\in\mathscr F~\text{and}~\bsy\in U.
$$
\item[(A3)] Gevrey regularity of $f$: there exist constants $C,\beta\geq 1$ and a vector of nonnegative numbers $\boldsymbol \rho=(\rho_j)_{j=1}^d \in \mathbb{R}_{\geq 0}^d$ such that%
$$
\|\partial_{\bsx}^{\bsnu} f\|_{L^\infty(\mathscr D)}\leq C(|\bsnu|!)^{\beta}\boldsymbol \rho^{\bsnu}\quad\text{for all}~\bsnu\in\mathscr F,
$$
where $f\!:\mathscr D\to\mathbb R$ is the source term on the right-hand side of~\eqref{DjKaOrSc:eq:nonpullback}.
\end{enumerate}
\end{addmargin}

\medskip

{\em Remark.} Without loss of generality, we may assume that the constants $C$ and $\beta$ in assumptions (A1), (A2), and (A3) coincide.

Analogously to~\cite{DjKaOrSc:harbrecht16} and~\cite[Lemma~3.2]{DjKaOrSc:hiptmair18}, assumption {\rm (A1)} implies that there exist constants $0<\sigma_{\min}\leq 1\leq \sigma_{\max}<\infty$ such that%
$$
\sigma_{\min}\leq \min\sigma(J(\bsx,\bsy))\leq \max \sigma(J(\bsx,\bsy))\leq \sigma_{\max}
~\text{ for all}~\bsx\in D_{\rm ref}~\text{and}~\bsy\in U,
$$
where $\sigma(J(\bsx,\bsy))$ denotes the set of all singular values of the matrix $J(\bsx,\bsy)$.

\medskip

The parametric weak formulation of~\eqref{DjKaOrSc:eq:poissonprob} can be stated as: for each $\bsy\in U$, find $u(\cdot,\bsy)\in H_0^1(D(\bsy))$ such that
\begin{align}
\int_{D(\bsy)}\nabla u(\bsx,\bsy)\cdot \nabla v(\bsx)\,{\rm d}\bsy=\int_{D(\bsy)}f(\bsx)v(\bsx)\,{\rm d}\bsx\quad\text{for all}~v\in H_0^1(D(\bsy)).\label{DjKaOrSc:eq:nonpullback}
\end{align}
However, for the analysis, it is convenient to consider the pullback of the solution to~\eqref{DjKaOrSc:eq:nonpullback} on the reference domain $D_{\rm ref}$ (cf. \cite{DjKaOrSc:harbrecht16}). We note that the pullback solution $\widehat u(\cdot,\bsy)\in H_0^1(D_{\rm ref})$ is related to $u(\cdot,\bsy)$ via
$$
\widehat u(\cdot,\bsy)=u(\boldsymbol V(\cdot,\bsy),\bsy)\quad\text{for all}~\bsy\in U,
$$
and in fact, the pullback solution can be identified as the solution to the following parametric weak formulation: for each $\bsy\in U$, find $\widehat u(\cdot,\bsy)\in H_0^1(D_{\rm ref})$ such that
$$
\int_{D_{\rm ref}}(A(\bsx,\bsy)\nabla \widehat u(\bsx,\bsy))\cdot \nabla \widehat v(\bsx)\,{\rm d}\bsx=\int_{D_{\rm ref}}f_{\rm ref}(\bsx,\bsy)\widehat v(\bsx)\,{\rm d}\bsx\quad\text{for all}~\widehat v\in H_0^1(D_{\rm ref}),
$$
where we define
$$
A(\bsx,\bsy):=(J(\bsx,\bsy)^{\rm T}J(\bsx,\bsy))^{-1}\det J(\bsx,\bsy)\quad\text{for all}~\bsx\in D_{\rm ref}~\text{and}~\bsy\in U
$$
and
$$
f_{\rm ref}(\bsx,\bsy):=f(\boldsymbol V(\bsx,\bsy))\det J(\bsx,\bsy)\quad\text{for all}~\bsx\in D_{\rm ref}~\text{and}~\bsy\in U.
$$

We shall also be interested in the dimensionally-truncated PDE solution $\widehat u_s(\cdot,\bsy)$ for $\bsy\in [-1/2,1/2]^s=:U_s$ corresponding to a finite-dimensional perturbation field $\boldsymbol V_s\!:\overline{D_{\rm ref}}\times U_s \to \mathbb R^d$, which we define as
$$
\boldsymbol V_s(\cdot,\bsy)=\boldsymbol V(\cdot,(y_1,\ldots,y_s,0,0,\ldots)).
$$

\subsection{Bayesian Inverse Problem}
\label{DjKaOrSc:subsec: Bayes}
We use the Bayesian statistical inversion paradigm to infer the realization of the uncertain domain based on measurements of the PDE forward model. Specifically, we assume that the unknown parameter $\bsy\in U_s$ has a uniform prior distribution $\mathcal U([-1/2,1/2]^s)$. We denote by $u_s$ the solution to the $s$-dimensional truncated problem  and in the following assume that $\cup_{\bsy \in U_s} D(\bsy) \subset \mathscr D$ also for all $s \in \mathbb{N}$. We note that this is not a restriction since alternatively one can define $\mathscr D^+ := \mathscr D \cup (\cup_{s \in \mathbb{N}} \cup_{\bsy \in U_s} D(\bsy)) $ and argue in $\mathscr D^+$. 

We consider the mathematical measurement model
\begin{equation}\label{DjKaOrSc:delta}
    \boldsymbol\delta=\mathcal G(\bsy)+\boldsymbol\eta,
\end{equation}
where $\boldsymbol\delta\in\mathbb R^k$ are the measurements, $\boldsymbol\eta\sim\mathcal N(0,\Gamma)$ is $k$-dimensional additive Gaussian noise with symmetric and positive definite covariance matrix $\Gamma\in\mathbb R^{k\times k}$, $\boldsymbol\eta$ is assumed to be independent of the process generating the observations, and $\mathcal G\!:U_s\to\mathbb R^k$ is the \emph{parameter-to-observation map}. We will consider the following particular form of observations that depends on fixed reference points and is given by
\begin{equation}\label{DjKaOrSc:observationOperator}
    \mathcal G(\bsy):={\mathcal{O}}(u_s(\cdot,\bsy), \bsy),\quad \bsy\in U_s,
\end{equation}
where
$\mathcal O\!:H_0^1(\mathscr D)\times U_s\to\mathbb R^k$ is the observation operator applied to the solution $u_s$ evaluated at $\boldsymbol V_s(\bsx_0, \bsy), \dots, \boldsymbol V_s(\bsx_{k-1}, \bsy)$, where the $\bsx_i \in D_{\rm ref}$, $i = 0, \dots, k-1$, are given fixed points. Note that this definition requires higher regularity of $u_s$, for example for $d \in \{2,3\}$ we will need that $u_s(\bsy) \in H^2(D(\bsy))$.

Bayes' formula can be used to express the posterior probability density function of the unknown parameter $\bsy$ conditioned on the measurements $\boldsymbol\delta$ as
$$
\pi(\bsy|\boldsymbol\delta)=\frac{\pi(\boldsymbol\delta|\bsy)\pi(\bsy)}{Z_s(\boldsymbol\delta)},\quad\bsy\in U_s,
$$
where $\pi(\bsy):=\mathbf 1_{U_s}(\bsy)$ for $\bsy\in U_s$ is the prior density, 
$$
\pi(\boldsymbol\delta|\bsy):=\exp\bigg(-\frac12 \|\boldsymbol\delta-\mathcal G(\bsy)\|_{\Gamma^{-1}}^2\bigg)
$$
is the likelihood, and
$$
Z_s(\boldsymbol\delta):=\int_{U_s}\pi(\boldsymbol\delta|\bsy)\,{\rm d}\bsy>0
$$
is the normalizing constant, which is always positive by the definition of $\pi(\boldsymbol\delta|\bsy)$. For the reconstruction of the unknown domain, we consider the Bayesian estimator of the expected value of the posterior distribution of $\bsV_s$ given data $\bsdelta$, which we denote by $\mathbb{E}_s^{\bsdelta}[\bsV_s](\bsx)$. It can be expressed as
$$
\mathbb{E}_s^{\bsdelta}[\bsV_s](\bsx):=\mathbb E_s[\boldsymbol V_s(\bsx,\cdot)|\boldsymbol\delta]=\frac{Z_s'(\boldsymbol\delta)}{Z_s(\boldsymbol\delta)},
$$
where $\bsx \in D_{\rm ref}$ and
$$
Z_s'(\boldsymbol\delta):=\int_{U_s} \boldsymbol V_s(\bsx,\bsy)\pi(\boldsymbol\delta|\bsy)\,{\rm d}\bsy.
$$

We also identify $$\mathbb E^{\boldsymbol\delta}[\boldsymbol V](\bsx):=\frac{Z'(\boldsymbol\delta)}{Z(\boldsymbol \delta)},$$ where
$$
Z(\boldsymbol\delta):=\int_U \exp\bigg(-\frac12 \|\boldsymbol\delta-\mathcal G(\bsy)\|_{\Gamma^{-1}}^2\bigg)\,{\rm d}\bsy>0
$$
and
$$
Z'(\boldsymbol\delta):=\int_U \boldsymbol V(\bsx,\bsy)\exp\bigg(-\frac12 \|\boldsymbol\delta-\mathcal G(\bsy)\|_{\Gamma^{-1}}^2\bigg)\,{\rm d}\bsy.
$$
The infinite-dimensional integrals are defined as $$\int_U F(\bsy)\,{\rm d}\bsy:=\lim_{s\to\infty}\int_{U_s}F(y_1,\ldots,y_s,0,0,\ldots)\,{\rm d}y_1\cdots\,{\rm d}y_s.$$

In this setting, both integrals  $Z'(\boldsymbol \delta)$ and $Z(\boldsymbol\delta)$ are well-defined since the integrands are continuous, so we can use the dominated convergence theorem together with the bound provided by assumption (A1) to ensure that the expectations exist. For more details, see~\cite{DjKaOrSc:kuoinfinite17}.

Both $Z_s'$ and $Z_s$ are high-dimensional integrals, which we want to approximate using a rank-1 QMC cubature rule. In order to find the appropriate weights that guarantee a dimension-independent error in this case, we will need to study the regularity of our QoI, i.e., the expected domain $\mathbb{E}_s^{\bsdelta}[\bsV_s](D_{\rm ref})$, with respect to the stochastic parameters. This is addressed later in Section \ref{DjKaOrSc:subsec:QMC error}.

\subsection{Quasi-Monte Carlo Methods}%
Let $F\!:U_s\to\mathbb R$ be a continuous function. In what follows, we will consider numerical approximations of $s$-dimensional integrals
$$
I(F):=\int_{U_s}F(\bsy)\,{\rm d}\bsy.
$$

The randomly shifted rank-1 QMC estimator of $I(F)$ is defined by (cf. \cite{DjKaOrSc:actanumer})
$$
Q_{n, s}(F):=\frac{1}{nR}\sum_{r=0}^{R-1} \sum_{\ell=0}^{n-1} F(\{\bst_\ell+\boldsymbol\Delta^{(r)}\}-\tfrac{\mathbf 1}{\mathbf 2}),
$$
where $\boldsymbol\Delta^{(0)},\ldots,\boldsymbol\Delta^{(R-1)}$ are i.i.d.~random shifts drawn from $\mathcal U([0,1]^s)$ and the cubature nodes are defined by
$$
\bst_\ell:=\bigg\{\frac{\ell\boldsymbol z}{n}\bigg\},\quad \ell\in\{0,\ldots,n-1\},
$$
where $\{\cdot\}$ denotes the component-wise fractional part and $\boldsymbol z\in\{1,\ldots,n-1\}^s$ is the {\em generating vector}.

In order to obtain the cubature error, we will assume that 
the integrand $F$ belongs to a weighted Sobolev space of dominating first order mixed smoothness $\mathcal W_{s,\boldsymbol\gamma}$ (cf. \cite{DjKaOrSc:kuonuyenssurvey}) endowed with the norm
$$
\|F\|_{s,\bsgamma}:=\bigg(\sum_{\setu\subseteq\{1,\ldots,s\}}\frac{1}{\gamma_{\setu}}\int_{U_{|\setu|}}\bigg(\int_{U_{s-|\setu|}}\frac{\partial^{|\setu|}}{\partial\bsy_{\setu}}F(\bsy)\,{\rm d}\bsy_{-\setu}\bigg)^2\,{\rm d}\bsy_{\setu}\bigg)^{1/2},
$$
where $\boldsymbol\gamma=(\gamma_{\setu})_{\setu\subseteq\{1,\ldots,s\}}$ is a sequence of positive weights, ${\rm d}\bsy_{\setu}:=\prod_{j\in\setu}{\rm d}y_j$ and ${\rm d}\bsy_{-\setu}:=\prod_{j\in\{1,\ldots,s\}\setminus\setu}{\rm d}y_j$.

The following result states that it is possible to construct generating vectors using a \emph{component-by-component (CBC)} algorithm~\cite{DjKaOrSc:cbc1,DjKaOrSc:actanumer,DjKaOrSc:cbc2} satisfying rigorous error bounds.
\begin{theorem}[{cf.~\cite[Theorem~5.1]{DjKaOrSc:kuonuyenssurvey}}]\label{DjKaOrSc:cbcbound}
Let $F \in \mathcal W_{s,\bsgamma}$ with weights $\bsgamma=(\gamma_{\setu})_{\setu\subseteq\{1,\ldots,s\}}$. For a prime number $n$, a randomly shifted lattice rule with $n$ points in $s$ dimensions can be constructed by a CBC algorithm such that for $R$ independent random shifts and for all $\lambda\in(1/2,1]$, there holds
$$
\sqrt{\mathbb E_{\boldsymbol\Delta}|I(F)-Q_{n, s}(F)|^2}\leq \frac{1}{\sqrt R}\bigg(\frac{1}{n-1}\sum_{\varnothing\neq\setu\subseteq\{1,\ldots,s\}}\gamma_{\setu}^{\lambda}\bigg(\frac{2\zeta(2\lambda)}{(2\pi^2)^\lambda}\bigg)^{|\setu|}\bigg)^{1/(2\lambda)}\|F\|_{s,\bsgamma},
$$
where the left-hand side is the rms error, $\mathbb E_{\boldsymbol\Delta}$ denotes the expected value with respect to the uniformly distributed random shifts over $[0,1]^s$ and $\zeta(x):=\sum_{k=1}^\infty k^{-x}$ is the Riemann zeta function for $x>1$.
\end{theorem}

In what follows, we will proceed to analyze the parametric regularity of our high-dimensional integral QoI. The parametric regularity bounds can be used to bound the weighted Sobolev norm $\|\cdot\|_{s,\bsgamma}$ appearing in the CBC error bound given in Theorem~\ref{DjKaOrSc:cbcbound} and the weights $\boldsymbol\gamma=(\gamma_{\setu})_{\setu\subseteq\{1,\ldots,s\}}$ can be chosen to control the error bound independently of the dimension $s.$ Finally, the weights $\boldsymbol\gamma$ can be used as inputs to the CBC algorithm to construct tailored lattice rules, which admit dimension-independent QMC convergence rates.

\section{Regularity and Error Analysis}\label{DjKaOrSc:sec:parametricanalysis}

In this section we analyze the total  error incurred in the numerical approximation of the 
expected domain given as the inverse problem through the Poisson equation. For the spatial discretization we will use the finite element method (FEM) and denote by $u_{s,h}$ the corresponding spatial approximation of the $s$-dimensional truncated problem. Given data $\bsdelta$, we write $\mathbb{E}_{s}^{\bsdelta}[\bsV_s](\bsx)$ for the posterior expectation of our QoI in the dimensionally-truncated problem, $\mathbb{E}_{s,h}^{\bsdelta}[\bsV_s](\bsx)$ for the posterior expectation of the dimensionally-truncated and discretized problem, and $\mathbb{E}_{s,h,n}^{\bsdelta}[\bsV_s](\bsx)$ 
for the QMC approximation of our QoI for the dimensionally-truncated and discretized problem. Similarly, we write $Z_s'(\bsdelta)$, $Z_{s,h}'(\bsdelta)$, and $Z_{s,h,n}'(\bsdelta)$. The total error $ \| \mathbb{E}^{\bsdelta}[\bsV](\bsx) - \mathbb{E}_{s,h,n}^{\bsdelta}[\bsV_s](\bsx) \|_{L^2(D_{\rm ref})}$ can then be decomposed into first terms
\begin{align*}
\| \mathbb{E}^{\bsdelta}[\bsV] - \mathbb{E}_{s,h,n}^{\bsdelta}[\bsV_s] \|_{L^2(D_{\rm ref})} &\leq
{\| \mathbb{E}^{\bsdelta}[\bsV] - \mathbb{E}_{s}^{\bsdelta}[\bsV_s]\|_{L^2(D_{\rm ref})}}
\\
&\quad +{\| \mathbb{E}_{s}^{\bsdelta}[\bsV_s] -\mathbb{E}_{s,h}^{\bsdelta}[\bsV_s]\|_{L^2(D_{\rm ref})}}
\\
&\quad +{\| \mathbb{E}_{s,h}^{\bsdelta}[\bsV_s] -\mathbb{E}_{s,h,n}^{\bsdelta}[\bsV_s] \|_{L^2(D_{\rm ref})}},
\end{align*}
where the first term is the dimension truncation error, the second one is the FEM error, and the last one is the QMC cubature error. We will bound each of these terms separately.

\subsection{Regularity of the Forward Problem}

Our starting point is the following result from \cite{DjKaOrSc:djurdjevac24}, which guarantees the Gevrey regularity with respect to the parameter $\bsy$ of the pullback solution in the reference domain and that we will later need for the truncation and QMC error analyses.

\begin{theorem}[cf.~{\cite{DjKaOrSc:djurdjevac24}}]  Let assumptions {\rm (A1)--(A3)} hold. Then, for all $\bsnu\in\mathscr F$ and $\bsy\in U$, there holds\label{DjKaOrSc:thm: forward regularity}
$$
\| \partial_{\boldsymbol{y}}^{\boldsymbol{\nu}} \widehat{u}(\cdot,\boldsymbol{y}) \|_{H_0^1(D_{\rm{ref}})} \leq c_1 \mathop{c_2^{|\boldsymbol{\nu}|}}(\mathop{|\boldsymbol{\nu}|!})^\beta\boldsymbol{b}^{\boldsymbol{\nu}}
$$
with
\begin{equation*}
    c_1 = 1 + \left(\frac{\sigma_{\max}}{\sigma_{\min}}\right)^d \frac{\sigma_{\max}^2 C_{D_{\rm ref}}|D_{\rm ref}|^{1/2}C}{2^\beta},
\end{equation*}
\begin{equation*}
    c_2 = \max \{\tilde{c}_1, \tilde{c}_2, \tilde{c}_3, \tilde{c}_4 \}^2 ((d^2)!)^\beta 2^{\beta(d^2+1)+1},
\end{equation*}
where 
$$
\tilde{c}_1 = \left(\frac{\sigma_{\max}}{\sigma_{\min}}\right)^d \frac{1}{\sigma_{\min}^2 ((d^2)!)^\beta}, \quad \tilde{c}_2 = 2\left(\frac{2^{\beta}C}{\sigma_{\min}}\right)^3,
$$
$$
\tilde{c}_3 = \frac{c_1 - 1}{\sigma_{\max}^{2} ((d^2)!)^\beta}, \quad \tilde{c}_4 =  \frac{(2^{\beta}C)^2}{\sigma_{\min}}\max\{1, \| \bsrho \|_{\ell^{1/\beta}}\},
$$
and $C_{D_{\rm{ref}}}$ is the Poincar\'e constant of $D_{\rm{ref}}$ while $|D_{\rm ref}|=\int_{D_{\rm ref}}{\rm d}\bsx$.
\end{theorem}

In applications, one does not typically observe the solution in the reference domain, but rather in a realization $D(\bsy)$. This motivates the study of observation operators considered in the sampled domain or depending on the solution $u$ in $D(\bsy)$. 

However, the pushforward solution $u$ does not directly inherit the regularity with respect to $\bsy$ from the pulled-back solution $\widehat{u}$, due to the spatial derivative of $\widehat{u}$ that appears from the chain rule. With our choice of the observation operator~\eqref{DjKaOrSc:observationOperator}, there holds
$$
\mathcal O(u(\cdot,\bsy),\bsy)=(u(\boldsymbol V(\bsx_0,\bsy),\bsy),\ldots,u(\boldsymbol V(\bsx_{k-1},\bsy),\bsy)),
$$
and hence we can exploit the identity $u(\boldsymbol V(\bsx,\bsy),\bsy)=\widehat u(\bsx,\bsy)$ to obtain 
$$
\partial_{\bsy}^{\bsnu}u(\boldsymbol V(\bsx,\bsy),\bsy)=\partial_{\bsy}^{\bsnu}\widehat u(\bsx,\bsy),
$$
which allows us to study the parametric regularity of our QoI in the sampled domain in terms of the known parametric regularity of the pulled-back solution.

\subsection{Error Analysis}

As already mentioned,  it is often necessary to truncate the perturbation field $\bsV$ to some stochastic dimension $s$. Hence, instead of the solution $u$ we consider the solution $u_s$ to the dimensionally-truncated problem. Further, in general the solution to the dimensionally-truncated problem cannot be calculated analytically and it is instead necessary to approximate it with a discretized solution. Finally, the application of the QMC rule introduces a numerical cubature error. In this section, we study these errors separately and obtain a convergence rate for the total error. 

\subsubsection{Truncation Error}We now consider the space of domains in $\mathscr D$ with the Hausdorff distance $d_H$, which makes $(\mathscr D, d_H)$ a metric space. We then define the map $S: \mathscr D \to H_0^1(\mathscr D)$ that maps a domain $D(\bsy)$ from $\mathscr D$  to the solution $u$ of the variational Poisson problem \eqref{DjKaOrSc:eq:nonpullback} in $D(\bsy)$. In what follows, we identify the solutions $u_s \in H_0^1(D(\bsy))$ with their zero extension in $H_0^1(\mathscr D)$.

\begin{lemma}
    \label{DjKaOrSc:lemma: continuity forward operator}
    Let $f \in L^2(\mathscr D)$ and $D_{\rm{ref}} \subset \mathbb{R}^d$ be a bounded reference domain. Then, the map $S: \mathscr D \to H_0^1(\mathscr D)$ is continuous.
\end{lemma}

\begin{proof}
    Since the result depends only on the distance of the domains and not on the perturbation field generating them, we do not specify whether $\bsy \in U$ or $\bsy \in U_s$ for some $s \in \mathbb{N}$. For the same reason, we also do not specify whether a solution $u$ solves \eqref{DjKaOrSc:eq:nonpullback} or its truncated version.
    
    Given a domain realization $\bsV(D_{\rm{ref}}, \bsy) = D(\bsy)$, we want to prove that for any sequence $(D_n(\bsy))_n \to D(\bsy)$, $(D_n(\bsy))_n \subset \mathscr D$ $\forall n \in \mathbb{N}$, we have $u_{n} \to u$ in $H_0^1(\mathscr D)$, where $u_{n}$ is the solution of the Poisson problem in $D_n(\bsy)$. 

    Without loss of generality, we consider an increasing sequence $D_n(\bsy) \nearrow D(\bsy)$, with $d_{H}(D_n(\bsy), D(\bsy)) < 1/n$, and define $z_n = u - u_{n} \in H_0^1(\mathscr D)$ and $\tilde{f} = f\cdot\mathbf{1}|_{D(\bsy)\backslash D_n(\bsy)}$. Our goal is then to prove $z_n \to 0$. By definition, $z_n$ satisfies
    \begin{align}   
        \begin{split}
        \label{DjKaOrSc:eq:solution-difference}
        - \Delta z_n = \tilde{f} \hspace{1.5em} &\text{in } D(\bsy), \\
        z_n = 0 \hspace{1.5em} &\text{on } \partial D(\bsy),
        \end{split}
    \end{align}
    and in particular, there holds $\Delta z_n = 0$ in $D_n(\bsy)$.

    Consider now $\mu_\epsilon * z_n =: z_{n}^\epsilon \in \mathcal C^\infty(\mathscr D)$. By the continuity of $z_{n}^\epsilon$ and the homogeneous boundary condition, we have that $z_{n}^\epsilon|_{\partial D_n(\bsy)} \xrightarrow[]{n \to \infty} 0$. Further, since $z_{n}^\epsilon$ is harmonic in $D_n(\bsy)$, it satisfies the maximum principle, so
    $$
        z_{n}^\epsilon|_{D(\bsy)} \xrightarrow[]{n \to \infty} 0,
    $$
    and by the convergence properties of mollifiers (cf. \cite{DjKaOrSc:GiTr2001}), we have
    $$
        z_{n}^\epsilon \xrightarrow[]{\epsilon \to \infty} z_{n} \xrightarrow[]{n \to \infty} 0 \text{ in } L^2(\mathscr D).
    $$

    On the other hand, $z_n$ is a weak solution of \eqref{DjKaOrSc:eq:solution-difference} if and only if
    $$
        \int_{D(\bsy)} \nabla z_n\cdot \nabla v \,{\rm d}\bsx= \int_{D(\bsy)} \tilde{f} v\,{\rm d}\bsx
    $$
    for all $v \in H_0^1(\mathscr D)$. In particular, choose $v = z_n$, and then there holds
    $$
        \int_{D(\bsy)} |\nabla z_n|^2 \,{\rm d}\bsx = \left| \int_{D(\bsy)} \tilde{f} z_n\, {\rm d}\bsx \right|  \leq \| \tilde{f} \|_{L^2(D(\bsy))} \| z_n \|_{L^2(D(\bsy))} \to 0,
    $$
    since both $\tilde{f}, z_n \to 0$ in $L^2(D(\bsy))$. Hence, we conclude that $\|z_n\|_{H_0^1(\mathscr D)} \xrightarrow[]{n \to \infty} 0$.\qed
\end{proof}

In order to obtain convergence rates for the dimension truncation error, we will need additional natural assumptions on the truncated perturbation field. 

\medskip

\noindent \textbf{Further assumptions on the perturbation field $\bsV$}

\begin{addmargin}[1.3em]{0em}
\begin{enumerate}
    \item[(A4)]  For all $\bsy \in U$, $\bsy_s \in U_s$, such that $\bsy|_{U_s} = \bsy_s$ there holds%
    \begin{equation*}
        \lim_{s \to \infty}\| \bsV_s(\cdot, \bsy_s) - \bsV(\cdot, \bsy) \|_{\mathcal C^1(\overline{D(\bsy)})} \to 0.
    \end{equation*}
    \item[(A5)] There exists $p\in(0,1)$, such that the sequence  $\bsb$ in (A3) satisfies $\boldsymbol b\in\ell^p(\mathbb N)$ and $b_1\geq b_2\geq \dots \geq 0$.%
\end{enumerate}
\end{addmargin}

\begin{theorem}
    \label{DjKaOrSc:thm: dimension-truncation}
    Suppose that assumptions {\rm (A1)--(A5)} hold. 
    Then
    $$
        \Bigg\| \int_{U} \widehat u(\cdot, \boldsymbol{y})\,{\rm d}\boldsymbol{y} - \int_{U_s} \widehat u_s(\cdot, \bsy)\,{\rm d}\bsy \Bigg\|_{H_0^1(D_{\rm ref})} \leq C_{\rm dim} s^{-2/p+1},
    $$
    where the constant $C_{\rm dim} > 0$ is independent of the dimension $s$. Further, let $\widehat{\mathcal{O}}: H_0^1(D_{\rm ref}) \to \mathbb{R}$ be an arbitrary linear and bounded functional. Then 
    $$
        \Bigg| \int_{U} \widehat{\mathcal{O}}(\widehat u(\cdot,\boldsymbol{y}))\,{\rm d}\boldsymbol{y} - \int_{U_s} \widehat{\mathcal{O}}(\widehat u_s(\cdot,\boldsymbol{y}))\,{\rm d}\boldsymbol{y} \Bigg| \leq C_{\rm dim} \|\widehat{\mathcal{O}}\| s^{-2/p+1},
    $$
    where $\| \widehat{\mathcal{O}} \|$ is the operator norm of $\widehat{\mathcal{O}}$ and the constant $C_{\rm dim} > 0$ is as above.
\end{theorem}

\begin{proof}
    The proof (see {\cite[Theorem 4.3]{DjKaOrSc:gk22}}) relies on the following condition: for a.e. $\bsy \in U$, $\bsy_s \in U_s$ such that $\bsy|_{U_s} = \bsy_s$ there holds
        \begin{equation}
            \| \widehat u(\cdot, \boldsymbol{y}) - \widehat u_s(\cdot, \bsy_s) \|_{H_0^1(D_{\rm ref})} \to 0 \text{ as } s \to \infty.
            \label{DjKaOrSc:eq:condition-truncation}
        \end{equation}
     
     By (A4), for $s \to \infty$ we have $\bsV_s(D_{\rm ref}, \bsy_s) \to \bsV(D_{\rm ref}, \bsy)$, and by Lemma \ref{DjKaOrSc:lemma: continuity forward operator} and the continuity of $\bsV$, this means that the corresponding solutions $u_s$, $u$ to \eqref{DjKaOrSc:eq:nonpullback} satisfy $\| u(\cdot, \bsy) - u_s(\cdot, \bsy_s) \|_{H_0^1(\mathscr D)} \to 0$. By the identity $\widehat u(\bsx, \bsy) = u(\bsV(\bsx, \bsy), \bsy)$ and assumption (A4), the convergence \eqref{DjKaOrSc:eq:condition-truncation} is satisfied.\qed
\end{proof}

Theorem \ref{DjKaOrSc:thm: dimension-truncation} together with our choice of observation operator yields an error estimate for the term ${\| \mathbb{E}^{\bsdelta}[\bsV] - \mathbb{E}_{s}^{\bsdelta}[\bsV_s]\|_{L^2(D_{\rm ref})}}$.
 
\subsubsection{Finite Element Error}
The following result controls the error when approximating $u_s$ by $u_{s,h}$, where $h$ is the mesh size, using a first-order FEM solver. In order to avoid the reference domain approximation error and to obtain higher regularity of the solution, we assume the additional regularity.

\medskip

\noindent \textbf{Further assumptions on the reference domain $D_{\rm ref}$ and perturbation field $\boldsymbol V$}
\begin{addmargin}[1.3em]{0em}
\begin{enumerate}
    \item [(A6)] The reference domain $D_{\rm ref}\subset\mathbb R^d$ is a convex and bounded polyhedron. %
    \item [(A7)] For each $\bsy\in U$, $\boldsymbol V(\cdot,\bsy)\!:\overline{D_{\rm ref}}\to\mathbb R^d$ is a $\mathcal C^2$-diffeomorphism.
\end{enumerate}
\end{addmargin}
\medskip

With these assumptions, by \cite[Theorem 3.2.1.2]{DjKaOrSc:Grisvard} we then have that for all $\bsy \in U_s$ $\widehat u_s(\cdot, \boldsymbol{y}) \in H^2(D_{\rm{ref}}) \cap H_0^1(D_{\rm{ref}})$, and thus $u_s(\cdot, \boldsymbol{y}) \in H^2(D(\bsy)) \cap H_0^1(D(\bsy))$. Hence, we can use the following result from \cite{DjKaOrSc:Qu2014}.

\begin{theorem}
    \label{DjKaOrSc:thm:FEM-error}
    Let assumptions {\rm (A1)--(A3)} and {\rm (A6)--(A7)} hold. Further, let $u_s(\cdot, \bsy) \in H_0^1(D(\boldsymbol{y}))$, $\bsy \in U_s$, be the exact solution of the variational problem \eqref{DjKaOrSc:eq:nonpullback} and $u_{s, h}$ its approximate solution obtained with a first-order FEM approximation with mesh size $h$. Moreover, let $u_s(\cdot, \bsy) \in \mathcal C^0(\overline{D(\boldsymbol{y})}) \cap H^{2}(D(\boldsymbol{y}))$. Then, the following \textit{a priori} error estimate in the norm of $L^2(D(\boldsymbol{y}))$ holds:
    $$
       \| u_s(\cdot, \bsy) - u_{s, h}(\cdot, \bsy) \|_{L^2(D(\boldsymbol{y}))} \leq C_{\rm FEM} h^{2} \| u_s(\cdot, \bsy) \|_{H^{2}(D(\boldsymbol{y}))},
    $$
    with $C_{\rm FEM} = C({\rm diam}(D(\boldsymbol{y})), \hat{K})$ being a constant independent of $h$ and $u$, where $\hat{K}$ is the reference element.
\end{theorem}

\begin{proof}
     See \cite[Theorem 4.7]{DjKaOrSc:Qu2014}, noting that $u_s(\cdot, \bsy) \in \mathcal C^0(\overline{D(\boldsymbol{y})}) \cap H^{2}(D(\boldsymbol{y}))$, for all $\bsy \in U_s$, since by the Kondrachov theorems one has the compact injection $H^{2}(D(\boldsymbol{y})) \hookrightarrow \mathcal C^0(\overline{D(\boldsymbol{y})})$ for $d < 4$.\qed
\end{proof}

\subsubsection{Quasi-Monte Carlo Error}
\label{DjKaOrSc:subsec:QMC error}

The following result states a suitable choice of {\em product and order dependent} (POD) weights ensuring dimension-independent QMC convergence for the approximation of the posterior mean.

\begin{theorem} \label{DjKaOrSc:thm: inverse exponential}
Let $\bsy \in U$. Under assumptions {\rm (A1)--(A3)} and {\rm (A6)--(A7)},
there holds
$$
\big|\partial_{\bsy}^{\bsnu}{\rm e}^{-\frac12 \|\boldsymbol\delta-\mathcal G(\bsy)\|_{\Gamma^{-1}}^2}\big|\leq c_3 c_4^{|\boldsymbol{\nu}|}  \mathop{(|\boldsymbol{\nu}|! )^\beta}\boldsymbol{b}^{\boldsymbol{\nu}},
$$
where
$$
c_3=\frac{1}{2^{\beta}}\cdot 3.47^k,\quad c_4=2^{\beta} c_1 c_2 \tau_{\min}^{-1/2},
$$
and $0<\tau_{\min}\leq 1$ is a lower bound on the smallest eigenvalue of $\Gamma$.
\end{theorem}
\begin{proof} See \cite[Lemma 5.3]{DjKaOrSc:ks24}.\qed \end{proof} 

\begin{theorem}
\label{DjKaOrSc:thm:regularity-expected-domain}
Let $\bsy \in U$. Under assumptions {\rm (A1)--(A3)} and {\rm (A6)--(A7)}, we have
$$
\big|\partial_{\bsy}^{\bsnu}\big(\boldsymbol V(\bsx,\bsy){\rm e}^{-\frac12 \|\boldsymbol\delta-\mathcal G(\bsy)\|_{\Gamma^{-1}}^2}\big)\big|\leq c_5c_{6}^{|\bsnu|} (\mathop{\left( |\boldsymbol{\nu}| + 1 \right)!})^\beta \boldsymbol{b}^{\boldsymbol{\nu}},
$$
where $c_5=Cc_3$ and $c_{6}=\max\{1,c_4\}$.
\end{theorem}
\begin{proof} 
Using Theorem \ref{DjKaOrSc:thm: inverse exponential} and the Leibniz product rule there holds
    \begin{align*}
         \Big| \partial_{\bsy}^{\boldsymbol{\nu}} \left(\bsV(\boldsymbol{x},\boldsymbol{y}) {\rm e}^{-\frac{1}{2}\| \boldsymbol{\delta} - \mathcal{G}(\boldsymbol{y}) \|_{\Gamma^{-1}}^2} \right) \Big|  
        & =  \Big| \sum_{ \boldsymbol{m} \leq \boldsymbol{\nu}} \binom{\boldsymbol{\nu}}{\boldsymbol{m}} \partial_{\boldsymbol y}^{\boldsymbol{m}} \boldsymbol V(\boldsymbol{x},\boldsymbol{y}) \partial_{\boldsymbol y}^{\boldsymbol{\nu}-\boldsymbol{m}} {\rm e}^{-\frac{1}{2}\| \boldsymbol{\delta} - \mathcal{G}(\boldsymbol{y}) \|_{\Gamma^{-1}}^2} \Big|
        \\
        & = \sum_{ \boldsymbol{m} \leq \boldsymbol{\nu}} \binom{\boldsymbol{\nu}}{\boldsymbol{m}} C (\mathop{|\boldsymbol{m}|!})^\beta\boldsymbol b^{\boldsymbol m} c_7c_8^{|\bsnu-\boldsymbol m|}((|\bsnu-\boldsymbol m|)!)^\beta \boldsymbol b^{\bsnu-\boldsymbol m}
        \\
        & \leq C c_3 \max\{1,c_4\}^{|\bsnu|}\boldsymbol b^{\bsnu}\sum_{ \boldsymbol{m} \leq \boldsymbol{\nu}} \binom{\boldsymbol{\nu}}{\boldsymbol{m}} (\mathop{|\boldsymbol{m}|!} \mathop{|\boldsymbol{\nu} - \boldsymbol{m}|!})^\beta\\
                & \leq Cc_3 \max\{1,c_4\}^{|\bsnu|}\boldsymbol b^{\bsnu}|\bsnu|!\sum_{\ell=0}^{|\bsnu|} (\mathop{\ell!} (\mathop{|\boldsymbol{\nu}| - \ell)!})^{\beta-1}
        \\
        & \leq Cc_3 \max\{1,c_4\}^{|\bsnu|}\boldsymbol b^{\bsnu}|\bsnu|!(|\bsnu|!)^{\beta-1}\sum_{\ell=0}^{|\bsnu|} 1\\
        & \leq Cc_3\max\{1,c_4\}^{|\bsnu|}\boldsymbol b^{\bsnu} (\mathop{ (|\boldsymbol{\nu}| + 1)!})^\beta,
    \end{align*}
which proves the assertion.\qed
\end{proof}

\begin{theorem} \label{DjKaOrSc:thm: weights}
For $s \in \mathbb{N}$, $n$ a prime number, and weights $\boldsymbol{\gamma} = (\gamma_{\mathfrak{u}})$, a randomly shifted lattice rule with $n$ points in $s$ dimensions can be constructed by a CBC algorithm such that the rms error for approximating the finite dimensional integral $Z_{s,h}'(\bsdelta)$ satisfies, for all $\lambda \in (1/2, 1],$
    \begin{equation*}
        \sqrt{\mathbb{E}_{\boldsymbol{\Delta}}[|Z_{s,h}'(\bsdelta) - Z_{s,h,n}'(\bsdelta)|^2]} \leq \frac{c_5 C_{\boldsymbol{\gamma}}(\lambda)}{(n-1)^{1/(2\lambda)}},
    \end{equation*}
    where $\mathbb{E}_{\boldsymbol{\Delta}}[\cdot]$ denotes the expectation with respect to the random shift which is uniformly distributed over $[0,1]^s$, and
    \begin{equation*}
        C_{\boldsymbol{\gamma}}(\lambda) := \left( \sum_{|{\mathfrak{u}}| < \infty} \gamma_{{\mathfrak{u}}}^\lambda \bigg[\frac{2\zeta(2\lambda)}{(2\pi^2)^\lambda}\bigg]^{|{\mathfrak{u}}|} \right)^{1/(2\lambda)}\left( \sum_{|{\mathfrak{u}}| < \infty} \frac{((|{\mathfrak{u}}|+1)!)^{2\beta} \prod_{j \in {\mathfrak{u}}} b_j^2}{\gamma_{{\mathfrak{u}}}} \right)^{1/2},
    \end{equation*}
    but $C_{\boldsymbol{\gamma}}(\lambda)$ is possibly infinite. Then the choice of weights
\begin{align*}
&\gamma_{\setu}=\bigg(((|\setu|+1)!)^\beta \prod_{j\in\setu}\frac{c_{6}b_j}{\sqrt{2\zeta(2\lambda)/(2\pi^2)^\lambda}}\bigg)^{2/(1+\lambda)},\quad\setu\subseteq\mathbb N,\\
&\lambda=\begin{cases}
\frac{p}{2-p}&\text{if}~p\in(\tfrac23,\tfrac{1}{\beta}),\\
\frac{1}{2-2\alpha}&\text{if}~p\in(0,\min\{\tfrac23,\tfrac{1}{\beta}\}],~p\neq \tfrac{1}{\beta},
\end{cases}
\end{align*}
for arbitrary $\alpha\in (0,1/2)$, minimizes $C_{\boldsymbol{\gamma}}(\lambda)$ and leads to
    $$
        C_{\boldsymbol{\gamma}}(\lambda) < \infty \hspace{1em} \text{ and } \hspace{1em} \sup_{s \geq 1}\| Z_{s,h}'(\bsdelta) \|_{\mathcal{W}_{s,\boldsymbol{\gamma}}} < \infty.
    $$
    However, $C_{\boldsymbol{\gamma}}(\frac{1}{2-2\delta}) \to \infty$ as $\alpha \to 0$, and $C_{\boldsymbol{\gamma}}(\frac{p}{2-p}) \to \infty$ as $p \to (2/3)^+$. In consequence, the error is of order
    $$
        \begin{cases}
            n^{-(1-\alpha)} &\text{ when } p \in (0, 2/3], \\
            n^{-(1/p-1/2)} &\text{ when } p \in (2/3, 1).
        \end{cases}
    $$
\end{theorem}

\begin{proof}
    The proof is carried out in a similar way to \cite[Theorem 6.4]{DjKaOrSc:kss12}.\qed
\end{proof}

\subsubsection{Total Error}
Since the bounds on the mixed derivatives of $Z_{s,h}'(\bsdelta$), i.e. the numerator of $\mathbb{E}_{s,h}^{\bsdelta}[\bsV_s](\bsx)$, dominate those of the denominator $Z_{s,h}(\bsdelta)$, this choice of weights guarantees the dimension-independent error bound in both cases. By usual arguments (see for example \cite{DjKaOrSc:BaKaLa24}), this in turn means that the estimation of the ratio satisfies the bound
$$
\sqrt{\mathbb{E}_{\boldsymbol{\Delta}}\left[\big\| \mathbb{E}_{s,h}^{\boldsymbol{\delta}}[\boldsymbol{V}_s] -\mathbb{E}_{s,h,n}^{\boldsymbol{\delta}}[\boldsymbol{V}_s] \big\|_{L^2(D_{\rm ref})}^2\right]} \lesssim n^{\max\{ -1/p + 1/2, -1+\alpha\}}.
$$

\begin{theorem}
    Let assumptions {\rm (A1)--(A7)} hold. With the choice of weights in Theorem \ref{DjKaOrSc:thm: weights}, the total error satisfies the following bound:
    $$
    \sqrt{\mathbb{E}_{\boldsymbol{\Delta}}\left[\big\| \mathbb{E}^{\boldsymbol{\delta}}[\boldsymbol{V}] -\mathbb{E}_{s,h,n}^{\boldsymbol{\delta}}[\boldsymbol{V}_s] \big\|_{L^2(D_{\rm ref})}^2\right]} \lesssim s^{-2/p+1} + h^2 + n^{\max\{ -1/p + 1/2, -1+\alpha\}}.
    $$
\end{theorem}

\begin{proof}
    The assertion follows from Theorems \ref{DjKaOrSc:thm: dimension-truncation}--\ref{DjKaOrSc:thm: weights}.\qed
\end{proof}

\section{Numerical Experiments}\label{DjKaOrSc:sec:numex}

In this section we present numerical experiments to verify the theoretical convergence rates in the setting presented above. We consider three methods: Monte Carlo (MC), QMC based on a generating vector $\bsz$ constructed with the weights derived in Theorem~\ref{DjKaOrSc:thm: weights}, leaving out the constant $c_6$ for numerical stability, and QMC with an off-the-shelf generating vector \cite[\texttt{lattice-32001-1024-1048576.3600}]{DjKaOrSc:Kuo24}. The software used for the numerical experiments is available at 
\begin{center}
\url{https://github.com/Morteu/QMC-for-Gevrey-Domain-UQ}
\end{center}

We consider examples where $D_{\rm ref} := \{ (x_1,x_2) \in \mathbb{R}^2 : x_1^2 + x_2^2 \leq 1 \}$ is the unit disk, and solve the variational problem \eqref{DjKaOrSc:eq:nonpullback} for $f(x_1, x_2) = 10\sin(x_1 x_2) - 5\cos(x_1+x_2)^2$ and a Gevrey regular (with Gevrey parameter $\beta = 2$) perturbation field $\bsV$ of the form $\boldsymbol{V}(\boldsymbol{x}, \boldsymbol{y}) = a(\boldsymbol{x}, \boldsymbol{y}) \boldsymbol{x}$, where
\begin{equation*}
a((x_1,x_2);\boldsymbol{y}) = 1 + \frac{6}{5} \sum_{j\geq 1} \frac{\cos\left(3 j \arctan2(x_1,x_2) - \pi/2 \right)}{j^{2.1}} \exp\left(-\frac{1}{\frac12 + y_j}\right).
\end{equation*}

As mentioned above, we set for our observation operator, 
$$
\mathcal{O}: H_0^1(\mathscr D) \times U_s \to \mathbb{R}^k, \quad (u,\bsy) \mapsto (u(\bsx_0, \bsy), \dots, u(\bsx_{k-1}, \bsy)),
$$
where we assume that $k$ fixed points are given in $D_{\rm ref}$ (see Figure \ref{DjKaOrSc:fig:point-evaluations} for $k = 5$).

\begin{figure}[!b]
    \centering
        \includegraphics[width=\textwidth]{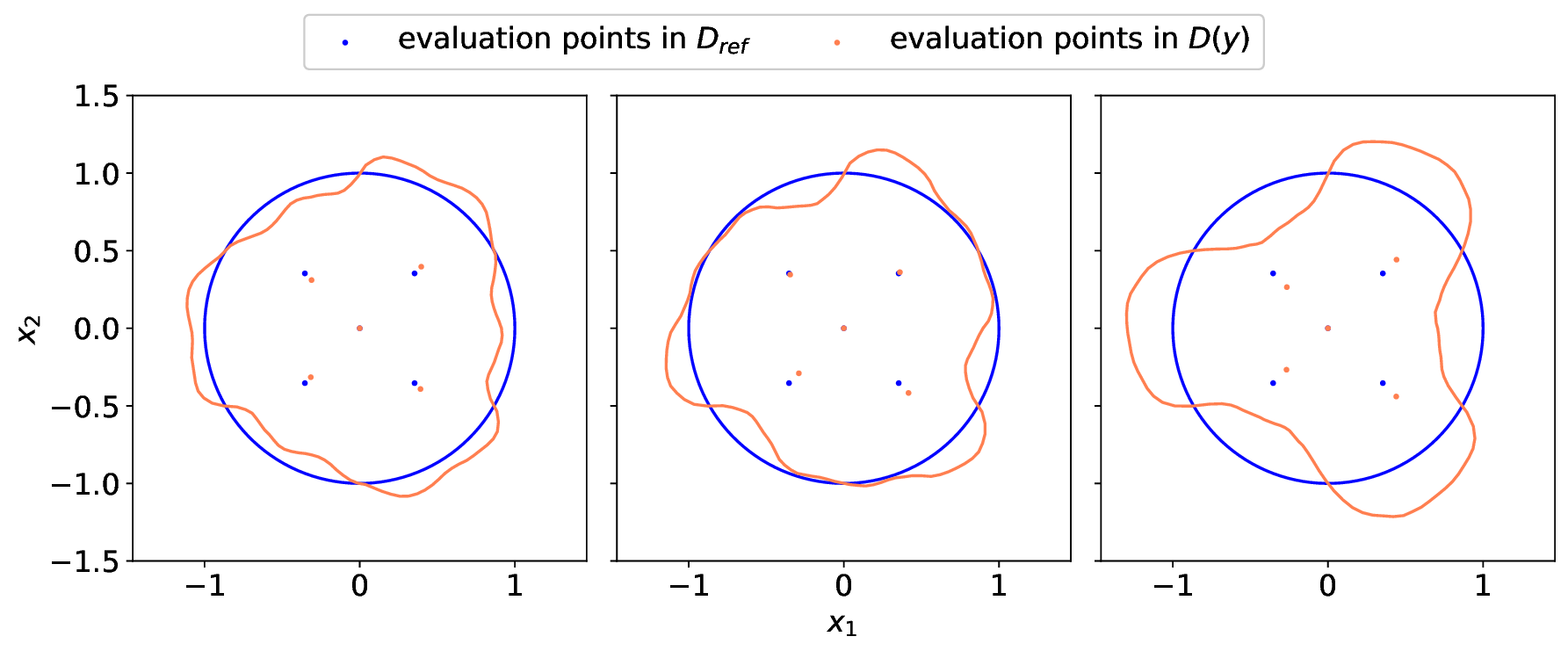}
    \caption{Fixed and transported evaluation points for three realizations, $k=5$.}
    \label{DjKaOrSc:fig:point-evaluations}
\end{figure}

The data $\bsdelta$ given by \eqref{DjKaOrSc:delta} is generated after sampling from a truncated $\bsV$ with a stochastic dimension $s_* = 200$ and discretizing the problem with a mesh size $h_* = 2^{-6}$. We then obtain $(\delta_0, \dots, \delta_{k-1}) = (u_{s_*,h_*}(\bsV_{s_*}(\bsx_0, \bsy)), \dots, u_{s_*,h_*}(\bsV_{s_*}(\bsx_{k-1}, \bsy)))$. Finally, we set the noise level $\eta$ to $10\%$ of the maximum absolute value in $\bsdelta$.

To reconstruct the uncertain domain, we set a stochastic dimension $s = 100$ and discretize the PDE solution with a mesh size $h = 2^{-5}$. To test the convergence of the rms error as predicted in Theorem \ref{DjKaOrSc:thm: weights}, we estimate the rms error using $R=8$ random shifts for prime values of $n$ ranging from $67$ to $128021$. 

\begin{figure}[!t]
    \centering
    \includegraphics[height=0.455\textwidth]{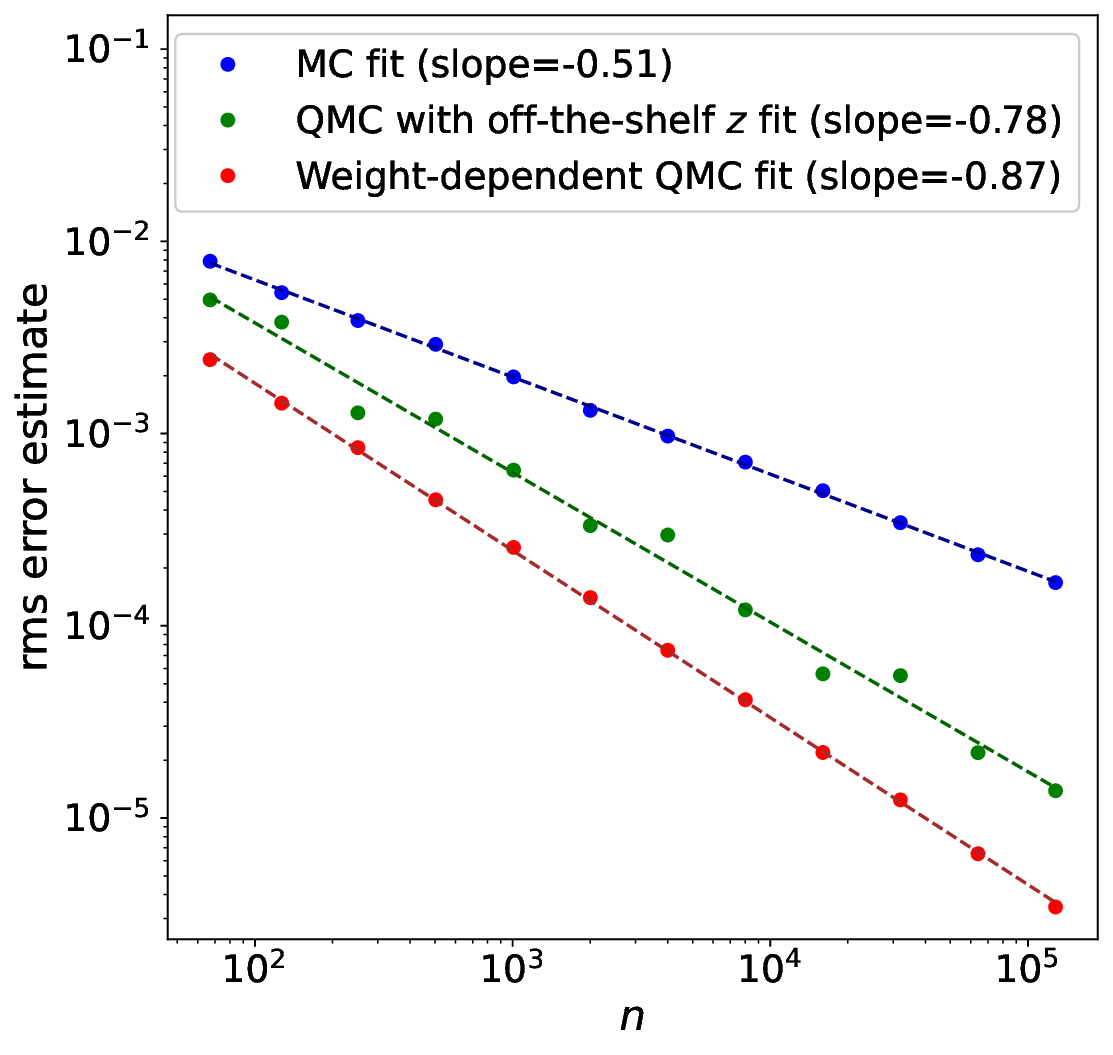}\includegraphics[height=0.46\textwidth]{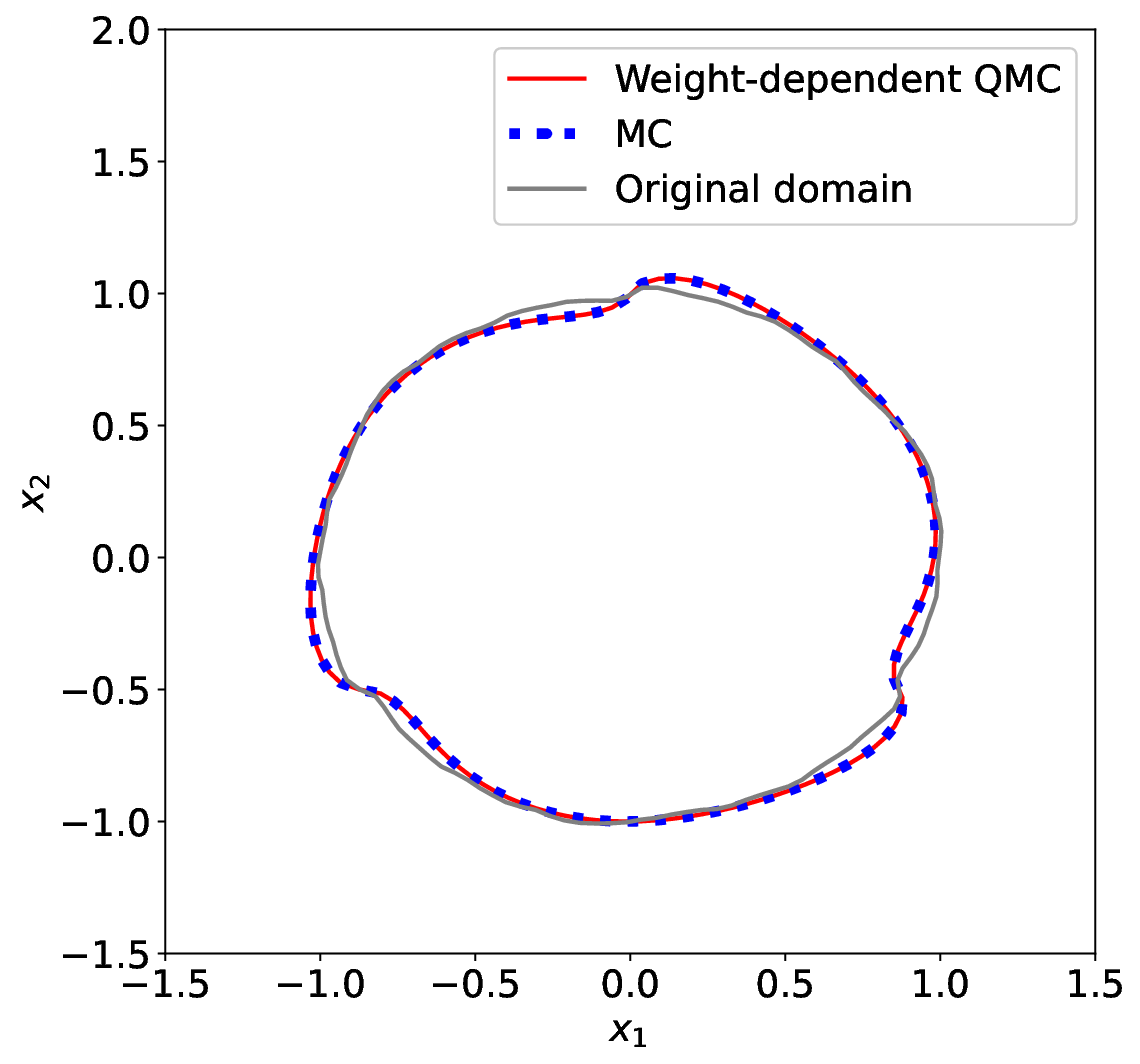}
    \caption{Results corresponding to the Bayesian inverse problem ($\eta = 10\% \| \bsdelta \|_{\infty}$, $k=5$). Left: computed rms errors for the ratio estimator with increasing $n$ for the three methods. Right: the reconstructed domains for MC and QMC with $n=128021$ vis-\`a-vis the ground truth.}
    \label{DjKaOrSc:fig:rms convergence}
\end{figure}

The results on the left-hand side of Figure \ref{DjKaOrSc:fig:rms convergence} show the rms error for the approximation $\mathbb{E}_{s,h,n}^{\bsdelta}[\bsV_s]$ with $k = 5$ observation points. We observe a roughly linear convergence rate for the QMC approximation with the weights derived in Theorem~\ref{DjKaOrSc:thm: weights}, and hence an almost doubled rate with respect to the MC approach. The QMC approximation with an off-the-shelf generating vector $\bsz$ also performs consistently better than MC, but it exhibits a higher variance.

On the right-hand side of Figure~\ref{DjKaOrSc:fig:rms convergence}, we present reconstructions for the weight-tailored QMC and the MC approximations. Due to the scaling of the parameters, both methods converge to the shown reconstructions already for small values of $n$. Nevertheless, the results indicate the consistency of the QMC approximation.

\section{Conclusions}\label{DjKaOrSc:sec:conclusions}

In this work, we have studied the application of randomly shifted rank-1 lattice rules to Bayesian shape inversion subject to Gevrey regular deformations of a reference domain with the Poisson equation as the forward model. Modeling the uncertain geometry using Gevrey regular perturbation fields covers a wider range of potential parameterizations than those covered by the affine or holomorphic frameworks, while also retaining a nearly optimal QMC convergence rate in the inverse setting as showcased by Theorem \ref{DjKaOrSc:thm: weights} and our numerical experiments. In addition, we showed that this setting leads to the optimal dimension truncation error rate as well as the standard FEM error rate.

The regularity analysis in this paper was presented for the pullback PDE solution. However, in many practical applications, it would be more natural to consider the pushforward PDE solution on the actual realization of the random or uncertain geometry instead. Our choice of the observation operator allows us to perform the experiments on the deformed domains while retaining the parametric regularity of the pullback PDE solution. Extending the regularity analysis for more general observation operators as well as more involved forward models is left for future work.

\section*{Acknowledgements}
Ana Djurdjevac, Max Orteu, and Claudia Schillings acknowledge funding by the Deutsche Forschungsgemeinschaft (DFG, German Research Foundation) CRC/TRR 388 ``Rough Analysis, Stochastic Dynamics and Related Fields'' – Project ID 516748464. The work of Vesa Kaarnioja was supported by the Research Council of Finland (Flagship of Advanced Mathematics for Sensing, Imaging and Modelling grant 359183). The authors would also like to thank the HPC Service of FUB-IT, Freie Universit\"at Berlin, for computing time.

\end{document}

%% file: macros.tex
\newcommand{\bsb}{{\boldsymbol{b}}}

\newcommand{\bst}{{\boldsymbol{t}}}

\newcommand{\bsx}{{\boldsymbol{x}}}
\newcommand{\bsy}{{\boldsymbol{y}}}
\newcommand{\bsz}{{\boldsymbol{z}}}

\newcommand{\bsV}{{\boldsymbol{V}}}

\newcommand{\bsgamma}{{\boldsymbol{\gamma}}}
\newcommand{\bsdelta}{{\boldsymbol{\delta}}}

\newcommand{\bsnu}{{\boldsymbol{\nu}}}

\newcommand{\bsrho}{{\boldsymbol{\rho}}}

\DeclareSymbolFont{bbold}{U}{bbold}{m}{n}
\DeclareSymbolFontAlphabet{\mathbbold}{bbold}

\newcommand{\setu}{{\mathfrak{u}}}

